\documentclass[12pt,a4paper]{article}

\usepackage{amsmath,amssymb,amsthm,float}
\usepackage[dvips]{graphicx}

\numberwithin{equation}{section}
\numberwithin{figure}{section}

\newtheorem{theorem}{Theorem}[section]
\newtheorem{lemma}[theorem]{Lemma}
\newtheorem{proposition}[theorem]{Proposition}
\newtheorem{remark}[theorem]{Remark}

\begin{document}

\title{\textbf{Graphical translating solitons \\for the inverse mean curvature flow \\and isoparametric functions}}
\author{Tomoki Fujii}
\date{Department of Mathematics, Tokyo University of Science}
\maketitle

\begin{abstract}
In this paper, we consider a translating soliton for the inverse mean curvature flow given as a graph of a function on a domain in a unit sphere whose level sets give isoparametric foliation.
First, we show that such function is given as a composition of an isoparametric function on the unit sphere and a function which is given as a solution of  a certain ordinary differential equation.
Further, we analyze the shape of the graphs of the solutions of the ordinary differential equation.
This analysis leads to the classification of the shape of such translating solitons for the inverse mean curvature flow.
\end{abstract}

\section{Introduction}

The author\cite{F} classified the shape of the translating soliton for the mean curvature flow given as a graph of a function on a domain in the unit sphere which is a composition of an isoparametric function and some function.
In this paper, we consider the case of the inverse mean curvature flow by the similar way.

Let $N$ be an $n$-dimensional Riemannian manifold.
Define the immersion $f$ of a domain $M\subset N$ into the product Riemannian manifold $N\times\mathbb{R}$ by $f(x)=(x,u(x)),~x\in M$ with a smooth function $u:M\to\mathbb{R}$ on $M$.
Also, denote the graph of $u$ (i.e, $f(M)$) by $\Gamma$.
For a $C^{\infty}$-family of $C^{\infty}$-immersions $\{f_t\}_{t\in I}$ of $M$ into $N\times\mathbb{R}$ ($I$ is an open interval including $0$) with $f_0=f$, as $M_t=f_t(M)$, $\{M_t\}_{t\in I}$ is called the inverse mean curvature flow starting from $\Gamma$ if $f_t$ satisfies 
\begin{equation}\label{eq:mcf}
\displaystyle\left(\frac{\partial f_t}{\partial t}\right)^{\bot_{f_t}}=-\frac{1}{\|H_t\|^2}H_t,
\end{equation}
where $H_t$ is the mean curvature vector field of $f_t$ and $(\bullet)^{\bot_{f_t}}$ is the normal component of $(\bullet)$ with respect to $f_t$.

Further, according to the definition of a soliton of the mean curvature flow by Hungerb\"{u}hler and Smoczyk\cite{HS}, we define a soliton of the inverse mean curvature flow.
Let $X$ be a Killing vector field on $N\times\mathbb{R}$ and $\{\phi_t\}_{t\in\mathbb{R}}$ be the one-parameter transformation associated to $X$ on $N\times\mathbb{R}$, that is, $\phi_t$'s are isometries and $\phi_t$ satisfies
\begin{equation*}
\displaystyle\frac{\partial \phi_t}{\partial t}=X\circ\phi_t,\quad
\phi_0=id_{N\times\mathbb{R}},
\end{equation*}
where $id_{N\times\mathbb{R}}$ is the identity map on $N\times\mathbb{R}$.
Then, the inverse mean curvature flow $\{M_t\}_{t\in I}$ is called a \textit{soliton for the inverse mean curvature flow with respect to $X$} if  $\widetilde{f}_t=\phi_t^{-1}\circ f_t$ satisfies
\begin{equation}
\left(\frac{\partial \widetilde{f}_t}{\partial t}\right)^{\bot_{\widetilde{f}_t}}=0.\label{eq:mcf-soliton}
\end{equation}
In the sequel, we call such soliton an \textit{$X$-soliton} simply.
In particular, when $X=(0,1)\in T(N\times\mathbb{R})=TN\oplus T\mathbb{R}$, we call the $X$-soliton a \textit{translating soliton}.

Compared with the mean curvature flow, the translating soliton for the inverse mean curvature flow is less studied.
For a translating soliton for the mean curvature flow, the existence of  the complete rotationally symmetric graphical translating soliton which is called bowl soliton is showed by Clutterbuck, Schn\"{u}rer and Schulze\cite{CSS} (Altschuler and Wu\cite{AW} had already showed the existence in the case $n=2$).
Also, they showed a certain type of stability for the bowl soliton and investigated the asymptotic expansion as the distance function $r$ in $\mathbb{R}^n$ approaches infinity because the bowl soliton is the graph of a function which is a composition of $r$ and the solution of a certain ordinary differential equation.
Further, Wang\cite{W} showed that the bowl soliton is the only convex translating soliton which is an entire graph.
Also, Spruck and Xiao\cite{SX} showed that the bowl soliton is the only complete translating soliton which is an entire graph.
Inspired by Clutterbuck, Schn\"{u}rer and Schulze\cite{CSS}, the author\cite{F} analyze the shape of the translating soliton for the mean curvature flow given as a graph of a function which is a composition of an isoparametric function on an $n$-dimensional unit sphere $\mathbb{S}^n$ and some function which is a solution of  a certain ordinary differential equation.
For a translating soliton for the inverse mean curvature flow, Drugan, Lee, and Wheeler\cite{DLW} gave a translating soliton in $\mathbb{R}^2$ which is the cycloid generated by a circle with radius $\frac{1}{4}$ and gave a tilted cycloid product as a translating soliton in $\mathbb{R}^3$.
Kim and Pyo\cite{KP1,KP2} showed the existence and classification of rotationally symmetric translating solitons in $\mathbb{R}^{n+1}$ and showed that there is no complete translating soliton for inverse mean curvature flow in $\mathbb{R}^{n+1}$.

In the main theorem of this paper, we consider the case where $N$ is the $n$-dimensional unit sphere $\mathbb{S}^n$ and $u$ is a composition of an isoparametric function $r$ on $\mathbb{S}^n$ and some function $V$.
Then, the level sets of the isoparametric function $r$ give compact isoparametric hypersurfaces of $\mathbb{S}^n$.
The isoparametric hypersurfaces of $\mathbb{S}^n$ has been studied by several authors.
M\"{u}nzner\cite{M} showed that the number $k$ of distinct principal curvatures of compact isoparametric hypersurfaces of  $\mathbb{S}^n$ is $1$, $2$, $3$, $4$ or $6$.
By Cartan\cite{C}, the isoparametric hypersurfaces in cases $k=1$, $2$, $3$ are classified.
Also, these hypersurfaces are homogeneous.
By the result of Dorfmeister and Neher\cite{DN} and Miyaoka\cite{Mi}, the isoparametric hypersurfaces in case $k=6$ are homogeneous.
Further, in case $k=4$, Ozeki and Takeuchi\cite{OT1,OT2} constructed non-homogeneous isoparametric hypersurfaces as the regular level sets of the restrictions of the Cartan-M\"{u}nzner polynomial functions to the sphere.

In this paper, we obtain the following theorem for the shape of the graph of $V$.

\begin{theorem}\label{thm:shape-graph}
Let $r$ be an isoparametric function on $\mathbb{S}^n$~$(n\ge2)$ and $V$ be a $C^{\infty}$-function on an interval $J\subset r(\mathbb{S}^n)$.
If the inverse mean curvature flow starting from the graph of the function $u=(V\circ r)\vert_{r^{-1}(J)}$ is a translating soliton, the shape of the the graph of $V$ is like one of  those defined by {\rm Figures \ref{Vex1}}$-${\rm \ref{Vex5}}.
\end{theorem}
\vspace{-0.15cm}
\begin{figure}[H]
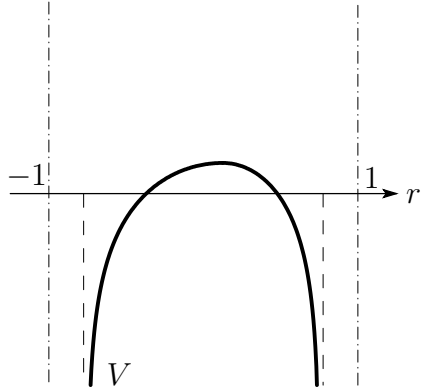

\centering
\scalebox{1.0}{\unitlength 0.1in%
}
\caption{The graph of $V$ (Type I)}
\label{Vex1}
\end{figure}

\begin{figure}[H]
\centering
\begin{minipage}{0.45\columnwidth}
\centering
\scalebox{1.0}{\unitlength 0.1in%
%
}
\caption{The graph of $V$ (Type II)}
\label{Vex2}
\end{minipage}
\hspace{0.5cm}
\begin{minipage}{0.45\columnwidth}
\centering
\scalebox{1.0}{\unitlength 0.1in%
%
}
\caption{The graph of $V$ (Type III)}
\label{Vex3}
\end{minipage}
\end{figure}

\begin{figure}[H]
\centering
\begin{minipage}{0.45\columnwidth}
\centering
\scalebox{1.0}{\unitlength 0.1in%
%
}
\caption{The graph of $V$ (Type IV)}
\label{Vex4}
\end{minipage}
\hspace{0.5cm}
\begin{minipage}{0.45\columnwidth}
\centering
\scalebox{1.0}{\unitlength 0.1in%
%
}
\caption{The graph of $V$ (Type V)}
\label{Vex5}
\end{minipage}
\end{figure}

\begin{remark}
For the $C^{\infty}$-function $V$ in Theorem \ref{thm:shape-graph}, define a $C^{\infty}$-function $\psi$ by $\psi(r)=k\sqrt{1-r^2}V'(r)$ and define $\psi_{min}$, $\psi_{max}$ by $\psi_{min}:=min_{r\in Dom(\psi)}\psi(r)$, $\psi_{max}:=max_{r\in Dom(\psi)}\psi(r)$.
Also, define a constant $R$ and functions $\eta_1$, $\eta_2$ on $(-1,a]\cup[b,1)$ by
\begin{align*}
&R:=
\begin{cases}
0\hspace{2.3cm}(k=1,3,6)\\
\displaystyle -1+\frac{km}{n-1}\quad(k=2,4),
\end{cases}\\
&\eta_1(r):=\frac{(n-1)(r-R)-\sqrt{((n-1)^2+4)r^2-2R(n-1)^2r+R^2(n-1)^2-4}}{2\sqrt{1-r^2}}\\
&\eta_2(r):=\frac{(n-1)(r-R)+\sqrt{((n-1)^2+4)r^2-2R(n-1)^2r+R^2(n-1)^2-4}}{2\sqrt{1-r^2}}.
\end{align*}
Here,  $m$ is the multiplicity of the smallest principal curvature of the isoparametric hypersurface defined by the level set of the isoparametric function $r$ in Theorem \ref{thm:shape-graph} and $a$, $b$ are defined by
\begin{align*}
a:=\frac{(n-1)^2R-2\sqrt{(n-1)^2(1-R^2)+4}}{(n-1)^2+4}\\
b:=\frac{(n-1)^2R+2\sqrt{(n-1)^2(1-R^2)+4}}{(n-1)^2+4}.
\end{align*}
If the graph of $V$ is like one defined by Figure \ref{Vex2}, then there exists $r_0\in(b,1)$ with $\psi_{min}=\eta_1(r_0)$ or $\psi_{min}=\eta_2(r_0)$.
If the graph of $V$ is like one defined by Figure \ref{Vex3}, then there exists $r_0\in(-1,a)$ with $\psi_{max}=\eta_1(r_0)$ or $\psi_{max}=\eta_2(r_0)$.
Also, we obtain the following table about the behavior of the graph of $\psi$ for each type of the graph of $V$.
Here, we define $x$, $y\in(-1,1)$ by $Dom(\psi)=(x,y)$.

\begin{table}[H]
\caption{The behavior of the graph of $\psi$}
\label{table:psi-behavior}
\begin{center}
\begin{tabular}{|c|c|c|c|c|}\hline
the graph of $V$ & $Im(\psi)$ & $\psi'$ & $r\downarrow x$ & $r\uparrow y$\\ \hline
Type I & $(-\infty,\infty)$ & $<0$ & $\infty$ & $-\infty$\\
Type II & $[\eta_i(r_0),\infty)$ & $-$ & $\infty$ & $\infty$\\
Type III & $(-\infty,\eta_i(r_0)]$ & $-$ & $-\infty$ & $-\infty$\\
Type IV & $[0,\infty)$ & $<0$ & $\infty$ & $0$\\
Type V & $(-\infty,0]$ & $<0$ & $0$ & $-\infty$\\ \hline
\end{tabular}
\end{center}
\end{table}
\end{remark}

\section{Proof of Theorem \ref{thm:shape-graph}}

Let $(N,g)$ be an $n$-dimensional Riemannian manifold and $u:M\to\mathbb{R}$ be a function on a domain $M\subset N$.
Denote the graph of $u$ by $\Gamma$.
Also, denote the gradient and Laplacian with respect to $g$ by $\nabla$ and $\Delta$ respectively.
Then, we have the following lemma about the soliton of the inverse mean curvature flow.
\begin{lemma}\label{lemma:graph-soliton-condition}
If the inverse mean curvature flow starting from $\Gamma$ is a translating soliton, $u$ satisfies
\begin{equation}
\Delta u+\|\nabla u\|^2+1-\frac{\nabla u(\|\nabla u\|^2)}{2(1+\|\nabla u\|^2)}=0.\label{eq:graph-soliton}
\end{equation}
Conversely, if $u$ satisfies {\rm (\ref{eq:graph-soliton})}, the family of the images $\{M_t\}_{t\in\mathbb{R}}$ definded by $f_t(x)=(x,u(x)+t),~x\in M$ and $M_t=f_t(M)$ is a translating soliton for the inverse mean curvature flow.
\end{lemma}

\begin{proof}
Define the immersion $f$ of $M$ into the product Riemannian manifold $N\times\mathbb{R}$ by $f(x)=(x,u(x)),~x\in M$ and define the Killing vector $X=(0,1)\in T(N\times\mathbb{R})=TN\oplus T\mathbb{R}$.
Denote the mean curvature vector field of $f$ by $H$.
According to Hungerb\"{u}hler and Smoczyk\cite{HS} in the case of a soliton for the mean curvature flow, if the inverse mean curvature flow starting from $\Gamma$ is translating soliton, we find that
\begin{equation}
\left(X\circ f\right)^{\bot_f}=-\frac{1}{\|H\|^2}H.\label{eq:soliton0}
\end{equation}
Let $(x^1,\cdots x^n,s)$ be local coordinates of $N\times\mathbb{R}$.
By $X=\frac{\partial}{\partial s}$ and $f(x)=(x,u(x)),~x\in M$, we find
\begin{align*}
&\left(X\circ f\right)^{\bot_f}=\frac{\partial}{\partial s}-\frac{1}{1+\|\nabla u\|^2}df(\nabla u)\\
&\frac{1}{\|H\|^2}H=\frac{1+\|\nabla u\|^2}{\Delta u-\frac{\nabla u(\|\nabla u\|^2)}{2(1+\|\nabla u\|^2)}}\left(\frac{\partial}{\partial s}-\frac{1}{1+\|\nabla u\|^2}df(\nabla u)\right).
\end{align*}
Therefore, we obtain that (\ref{eq:soliton0}) is equivalent to (\ref{eq:graph-soliton}).

Conversely, if $u$ satisfies {\rm (\ref{eq:graph-soliton})}, we find that $f$ satisfies (\ref{eq:soliton0}).
Then, for the one-parameter transformation $\{\phi_t\}_{t\in\mathbb{R}}$ associated to $X$ on $N\times\mathbb{R}$, since $\phi_t$'s are isometries and $f$ satisfies {\rm (\ref{eq:soliton0})}, we find that $f_t=\phi_t\circ f$ satisfies
\begin{align*}
\left(\frac{\partial f_t}{\partial t}\right)^{\bot_{f_t}}+\frac{1}{\|H_t\|^2}H_t&=d\phi_t\left((X\circ f)^{\bot_f}+\frac{1}{\|H\|^2}H\right)\\
&=0,
\end{align*}
and $\{f_t\}_{t\in \mathbb{R}}$ satisfies (\ref{eq:mcf}).
So, $\{M_t\}_{t\in\mathbb{R}}$ is the inverse mean curvature flow.
Further, we find that $f_t$ satisfies (\ref{eq:mcf-soliton}) from $\phi_t^{-1}\circ f_t=f$.
Therefore, $\{M_t\}_{t\in\mathbb{R}}$ is a translating soliton.
Then, we have $f_t(x)=(x,u(x)+t),~x\in M$.
\end{proof}

We consider the case where $u$ is a composition of an isoparametric function $r$ and some function $V$.
A non-constant $C^{\infty}$-function $r:N\to\mathbb{R}$ is called an isoparametric function if there exist $C^{\infty}$-functions $\alpha,\beta$ such that
\begin{equation*}
\begin{cases}
\|\nabla r\|^2=\alpha\circ r\\
\Delta r=\beta\circ r.
\end{cases}
\end{equation*}
Then, the level set of $r$ with respect to a regular value is called an isoparametric hypersurface.
For Lemma \ref{lemma:graph-soliton-condition}, considering the case where $u$ is the composition of the isoparametric function $r : N\to\mathbb{R}$ and some function $V$, we obtain the following proposition.
\begin{proposition}
Let $r:N\to\mathbb{R}$ be an isoparametric function on $N$.
If  the inverse mean curvature flow starting from $\Gamma$ is a translating soliton and if there exists a $C^{\infty}$-function $V$ on an interval $J\subset r(N)$ such that $u=(V\circ r)\vert_{r^{-1}(J)}$, the function $V$ satisfies
\begin{equation}
2\alpha V''+2\alpha^2V'^4+\alpha(2\beta-\alpha') V'^3+4\alpha V'^2+2\beta V'+2=0,\label{eq:isopara-soliton}
\end{equation}
where $'$ denotes derivative on $J$ and $\alpha,\beta$ are $C^{\infty}-$functions which satisfy $\|\nabla r\|^2=\alpha\circ r,~\Delta r=\beta\circ r$.\label{eq:isopara-graph-soliton}
Conversely, if $V$ satisfies {\rm (\ref{eq:isopara-soliton})}, the family of the images $\{M_t\}_{t\in\mathbb{R}}$ defined by $f_t(x)=(x,(V\circ r)(x)+t),~x\in M$ and $M_t=f_t(M)$ is the translating soliton for the inverse mean curvature flow.
\end{proposition}
\begin{proof}
From (\ref{eq:graph-soliton}), we have
\begin{equation*}
2(1+\|\nabla u\|^2)(\Delta u+\|\nabla u\|^2+1)-\nabla u(\|\nabla u\|^2)=0.
\end{equation*}
By $u=V\circ r$, we find
\begin{equation*}
\begin{split}
\|\nabla u\|^2&=\left(\alpha V'^2\right)\circ r,\\
\nabla u(\|\nabla u\|^2)&=\left(\alpha V'^2\left(2\alpha V''+\alpha'V'\right)\right)\circ r,\\
\Delta u&=\left(\alpha V''+\beta V'\right)\circ r.
\end{split}
\end{equation*}
Therefore, (\ref{eq:graph-soliton}) is reduced to the following equation
\begin{equation*}
2(1+\alpha V'^2)(\alpha V''+\beta V'+\alpha V'^2+1)-\alpha V'^2(\alpha'V'+2\alpha V'')=0.
\end{equation*}
From this equation, we obtain (\ref{eq:isopara-soliton}).
\end{proof}

In the case where $N$ is the $n$-dimensional unit spere $\mathbb{S}^n$, M\"{u}nzner\cite{M,M2} show the following theorem for an isoparametric function $r$ on $\mathbb{S}^n$.
\begin{theorem}\label{thm:Mun}
{\rm (M\"{u}nzner\cite{M,M2})}
(i) Let $r$ be an isoparametric function on $\mathbb{S}^n$.
Then, $r$ satisfies
\begin{equation}\label{isopara-sphere}
\begin{cases}
\|\nabla r\|^2=k^2(1-r^2)\\
\displaystyle\Delta r=\frac{m_2-m_1}{2}k^2-k(n+k-1)r.
\end{cases}
\end{equation}
By the first equation of {\rm (\ref{isopara-sphere})}, we find that $r(\mathbb{S}^n)=[-1,1]$.
Here, we consider that the isoparametric hypersurface defined by the level set of $r$ has $k$ distinct principal curvatures $\lambda_1>\cdots>\lambda_k$ with respective multiplicities $m_1,\cdots,m_k$.

(ii) The number $k$ of distinct principal curvatures is $1$, $2$, $3$, $4$ or $6$.

(iii) If $k=1$, $3$, $6$, then the mulitiplicities are equal.
If $k=2$, $4$, then there are at most two distinct multiplicities $m_1$, $m_2$.
\end{theorem}

In the sequel, we assume that $N$ is the $n$-dimensional unit sphere $\mathbb{S}^n$ ($n\ge2$) and  $u=(V\circ r)\vert_{r^{-1}(J)}$ with an isoparametric function $r : \mathbb{S}^n\to\mathbb{R}$ and a $C^{\infty}$-function $V$ on interval $J\subset r(\mathbb{S}^n)=[-1,1]$.
From Theorem \ref{thm:Mun} (iii), we find
\begin{equation*}
m_2-m_1=
\begin{cases}
0\quad \hspace{1.9cm}(k=1,3,6)\\
2(m_2-\frac{n-1}{k})\quad(k=2,4).
\end{cases}
\end{equation*}
Therefore, substituting $\alpha$ and $\beta$ in Theorem \ref{thm:Mun} (\ref{isopara-sphere}) for the equation (\ref{eq:isopara-soliton}), we obtain
\begin{align}
V''(r)=&-k^2(1-r^2)V'(r)^4+k((n-1)(r-R))V'(r)^3-2V'(r)^2\nonumber\\
&+\frac{(n+k-1)r-(n-1)R}{k(1-r^2)}V'(r)-\frac{1}{k^2(1-r^2)},~~~r\in (-1,1),\label{eq:isopara-soliton-sphere2}
\end{align}
where $R\in(-1,1)$ is the constant defined by
\begin{equation*}
R:=
\begin{cases}
0\hspace{2.3cm}(k=1,3,6)\\
\displaystyle -1+\frac{km_2}{n-1}\quad(k=2,4).
\end{cases}
\end{equation*}
Here, we note that $m_2$ is equal to the multiplicity of the smallest principal curvature of the isoparametric hypersurface defined by the level set of $r$ in the case $k=2,4$.
The local existence of the solution $V$ of (\ref{eq:isopara-soliton-sphere2}) is clear.
To prove Theorem \ref{thm:shape-graph}, we consider the graph of the solution $V$ of (\ref{eq:isopara-soliton-sphere2}).
Define $\psi(r)=k\sqrt{1-r^2}V'(r)$.
Then, the equation (\ref{eq:isopara-soliton-sphere2}) is reduced to
\begin{equation}
\psi'(r)=-\frac{1}{k(1-r^2)}\left(\psi(r)^2+1\right)\hspace{-0.1cm}\left(\sqrt{1-r^2}\psi(r)^2-(n-1)(r-R)\psi(r)+\sqrt{1-r^2}\right).\label{eq:isopara-soliton-sphere3}
\end{equation}
Therefore, to obtain the behavior of the graph of $V$, first we consider the behavior of the solution $\psi$ of (\ref{eq:isopara-soliton-sphere3}).
Define $\eta_1(r)$ and $\eta_2(r)$ by
\begin{align*}
&\eta_1(r):=\frac{(n-1)(r-R)-\sqrt{((n-1)^2+4)r^2-2R(n-1)^2r+R^2(n-1)^2-4}}{2\sqrt{1-r^2}}\\
&\eta_2(r):=\frac{(n-1)(r-R)+\sqrt{((n-1)^2+4)r^2-2R(n-1)^2r+R^2(n-1)^2-4}}{2\sqrt{1-r^2}}.
\end{align*}
Also, define $a$, $b\in(-1,1)$ ($a<b$) by
\begin{align*}
a:=\frac{(n-1)^2R-2\sqrt{(n-1)^2(1-R^2)+4}}{(n-1)^2+4}\\
b:=\frac{(n-1)^2R+2\sqrt{(n-1)^2(1-R^2)+4}}{(n-1)^2+4}.
\end{align*}
Then, we find $a<R<b$ and obtain the following lemma.
\begin{lemma}\label{thm:sign}
~
\begin{itemize}
\item[(i)] When $r\in(-1,a]\cup[b,1)$,
\begin{itemize}
\item [(a)] if $\eta_1(r)<\psi(r)<\eta_2(r)$, then $\psi'(r)>0$,
\item [(b)] if $\psi(r)=\eta_1(r)$ or $\psi(r)=\eta_2(r)$, then $\psi'(r)=0$,
\item [(c)] if $\psi(r)<\eta_1(r)$ or $\psi(r)>\eta_2(r)$, then $\psi'(r)<0$.
\end{itemize}
\item[(ii)] When $r\in(a,b)$, $\psi'(r)<0$.
\end{itemize}
\end{lemma}

\begin{proof}
Define $A(x,r)$ and $B(r)$ by
\begin{align*}
&A(x,r):=\sqrt{1-r^2}x^2-(n-1)(r-R)x+\sqrt{1-r^2},~~~(x,r)\in\mathbb{R}\times(-1,1)\\
&B(r):=((n-1)^2+4)r^2-2(n-1)^2Rr+(n-1)^2R^2-4,~~~r\in(-1,1).
\end{align*}
Then, we have
\begin{align*}
&A(x,r)=\sqrt{1-r^2}\left(x-\frac{(n-1)(r-R)}{2\sqrt{1-r^2}}\right)^2-\frac{1}{4\sqrt{1-r^2}}B(r)\\
&B(r)=\left((n-1)^2+4\right)\left(r-\frac{(n-1)^2R}{(n-1)^2+4}\right)^2-\frac{1}{(n-1)^2+4}\left(4(n-1)^2(1-R^2)+16\right).
\end{align*}
Therefore, we find that if $r\in(-1,a]\cup[b,1)$, then $B(r)>0$, if $r\in(a,b)$, then $B(r)<0$, and if $r\in\{a,b\}$, then $B(r)=0$.
Further, we find that when $r\in(-1,a]\cup[b,1)$, if $x\in(\eta_1(r),\eta_2(r))$, then $A(x,r)<0$, if $x\in(-\infty,\eta_1(r))\cup(\eta_2(r),\infty)$, then $A(x,r)>0$, and if $x\in\{\eta_1(r),\eta_2(r)\}$, then $A(\psi,r)=0$.
Also, when $r\in(a,b)$, we find that $A(x,r)>0$.
Since the equation (\ref{eq:isopara-soliton-sphere3}) is reduced to
\begin{equation*}
\psi'(r)=-\frac{1}{k(1-r^2)}\left(\psi(r)^2+1\right)A(\psi(r),r),
\end{equation*}
we obtain the statement of this lemma.
\end{proof}

\begin{figure}[H]
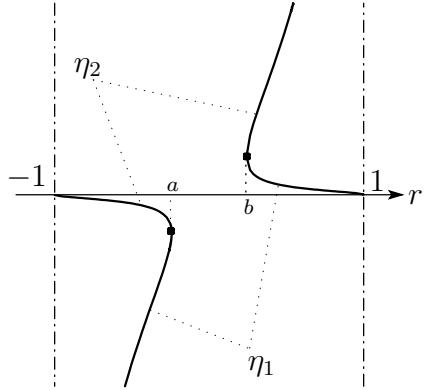

\centering
\scalebox{1.0}{\unitlength 0.1in%
}
\caption{The graph of $\eta_1$ and $\eta_2$}
\label{eta}
\end{figure}

For the behavior of the graph of the solution $\psi$ of (\ref{eq:isopara-soliton-sphere3}), we obtain following lemmata.

\begin{lemma}\label{thm:psi-shape1}
If there exists $r_0\in(-1,a]$ with $\psi(r_0)<\eta_1(r_0)$, or if there exists $r_0\in(a,1)$ with $\psi(r_0)<0$, then there exists $r_1\in(r_0,1)$ such that 
\begin{equation*}
\lim_{r\uparrow r_1}\psi(r)=-\infty.
\end{equation*}
\end{lemma}

\begin{proof}
When $r>r_0$, we find $\psi'(r)<0$ and $\psi(r)<\psi(r_0)$.
Then, we have
\begin{align*}
\psi'(r)&=-\frac{1}{k(1-r^2)}\left(\psi(r)^2+1\right)\left(\sqrt{1-r^2}\psi(r)^2-(n-1)(r-R)\psi(r)+\sqrt{1-r^2}\right)\\
&<-\frac{1}{k(1-r^2)}\left(\psi(r)^2+1\right)\left((1+\psi(r_0)^2)\sqrt{1-r^2}-\psi(r_0)(n-1)(r-R)\right).
\end{align*} 
Therefore, we find
\begin{equation*}
\frac{\psi'(r)}{1+\psi(r)^2}<-\frac{1+\psi(r_0)^2}{k\sqrt{1-r^2}}+\frac{\psi(r_0)(n-1)r}{k(1-r^2)}-\frac{\psi(r_0)(n-1)R}{k(1-r^2)}.
\end{equation*}
Integrating from $r_0$ to $r$, we have
\begin{align*}
\arctan{\psi(r)}<&-\frac{1+\psi(r_0)^2}{k}\arcsin{r}-\frac{\psi(r_0)(n-1)}{2k}\log{(1-r^2)}-\frac{\psi(r_0)(n-1)R}{2k}\log{\frac{1+r}{1-r}}\\
&+\frac{1+\psi(r_0)^2}{k}\arcsin{r_0}+\frac{\psi(r_0)(n-1)}{2k}\log{(1-r_0^2)}+\frac{\psi(r_0)(n-1)R}{2k}\log{\frac{1+r_0}{1-r_0}}\\
&+\arctan{\psi(r_0)}=:h_1(r).
\end{align*}
Then, $h_1$ is decreasing on $(r_0,1)$ and $h_1(r_0)=\arctan{\psi(r_0)}$, $\lim_{r\uparrow 1}h_1(r)=-\infty$.
Therefore, there exists $\overline{r}_1\in(r_0,1)$ with $h_1(\overline{r}_1)=-\frac{\pi}{2}$ and
\begin{equation*}
\psi(r)<\tan{h_1(r)}\rightarrow-\infty\quad(~r~\uparrow~\overline{r}_1~)
\end{equation*}
\end{proof}

\begin{figure}[H]
\centering
\scalebox{1.0}{\unitlength 0.1in%
}
\caption{The behavior of the graph of $\psi$ in Lemma \ref{thm:psi-shape1}}
\label{eta}
\end{figure}

\begin{lemma}\label{thm:psi-shape2}
If there exists $r_0\in\left(b,1\right)$ with $\eta_1(r_0)<\psi(r_0)<\eta_2(r_0)$, then
\begin{equation*}
\lim_{r\uparrow 1}\psi(r)=\infty.
\end{equation*}
\end{lemma}

\begin{proof}
Assume that there exists a constant $C>0$ such that $\psi(r)<C$ for all $r\in(r_0,1)$.
Define $\eta_3(r)$ by
\begin{equation*}
\eta_3(r):=\frac{1}{2}(\eta_1(r)+\eta_2(r))=\frac{(n-1)(r-R)}{\sqrt{1-r^2}}.
\end{equation*}
Then, there exists $\overline{r}_0\in(r_0,1)$ such that $\psi(\overline{r}_0)<\psi(r)<\eta_3(r)$ for all $r\in(\overline{r}_0,1)$.
Therefore, we have
\begin{align*}
\psi'(r)&=-\frac{1}{k(1-r^2)}\left(\psi(r)^2+1\right)\left(\sqrt{1-r^2}\psi(r)^2-(n-1)(r-R)\psi(r)+\sqrt{1-r^2}\right)\\
&>-\frac{1}{k(1-r^2)}\left(\psi(r)^2+1\right)\left((1+\psi(\overline{r}_0)^2)\sqrt{1-r^2}-\psi(\overline{r}_0)(n-1)(r-R)\right).
\end{align*}
Then, we find
\begin{equation*}
\frac{\psi'(r)}{1+\psi(r)^2}>-\frac{1+\psi(\overline{r}_0)^2}{k\sqrt{1-r^2}}+\frac{\psi(\overline{r}_0)(n-1)r}{k(1-r^2)}-\frac{\psi(\overline{r}_0)(n-1)R}{k(1-r^2)}.
\end{equation*}
Integrating from $\overline{r}_0$ to $r$, we have
\begin{align*}
\arctan{\psi(r)}>&-\frac{1+\psi(\overline{r}_0)^2}{k}\arcsin{r}-\frac{\psi(\overline{r}_0)(n-1)}{2k}\log{(1-r^2)}-\frac{\psi(\overline{r}_0)(n-1)R}{2k}\log{\frac{1+r}{1-r}}\\
&+\frac{1+\psi(\overline{r}_0)^2}{k}\arcsin{\overline{r}_0}+\frac{\psi(\overline{r}_0)(n-1)}{2k}\log{(1-\overline{r}_0^2)}+\frac{\psi(\overline{r}_0)(n-1)R}{2k}\log{\frac{1+\overline{r}_0}{1-\overline{r}_0}}\\
&+\arctan{\psi(\overline{r}_0)}=:h_2(r).
\end{align*}
Then, $h_2$ is increasing on $(\overline{r}_0,1)$ and $h_2(\overline{r}_0)=\arctan{\psi(\overline{r}_0)}$, $\lim_{r\uparrow 1}h_2(r)=\infty$.
Therefore, there exists $\overline{r}_1\in(r_0,1)$ with $h_2(\overline{r}_1)=\frac{\pi}{2}$ and
\begin{equation*}
\psi(r)>\tan{h_2(r)}\rightarrow\infty\quad(~r~\uparrow~\overline{r}_1~)
\end{equation*}
This contradicts the assumption that $\psi(r)<C$ for all $r\in(r_0,1)$.
\end{proof}

\begin{figure}[H]
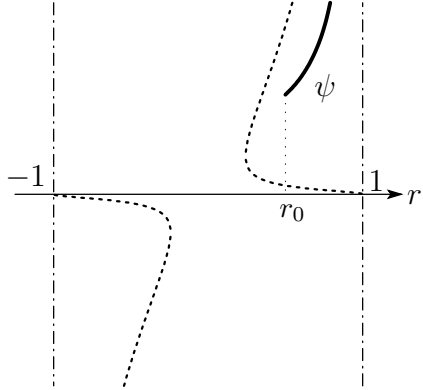

\centering
\scalebox{1.0}{\unitlength 0.1in%
}
\caption{The behavior of the graph of $\psi$ in Lemma \ref{thm:psi-shape2}}
\label{eta}
\end{figure}

By proofs similar to Lemma \ref{thm:psi-shape1} and Lemma \ref{thm:psi-shape2}, we obtain the following lemmata.

\begin{lemma}\label{thm:psi-shape3}
If there exists $r_0\in(-1,b]$ with $\psi(r_0)>0$ or if there exists $r_0\in(b,1)$ with $\psi(r)>\eta_2(r)$, then there exists $r_1\in(-1,r_0)$ such that 
\begin{equation*}
\lim_{r\downarrow r_1}\psi(r)=\infty.
\end{equation*}
\end{lemma}

\begin{figure}[H]
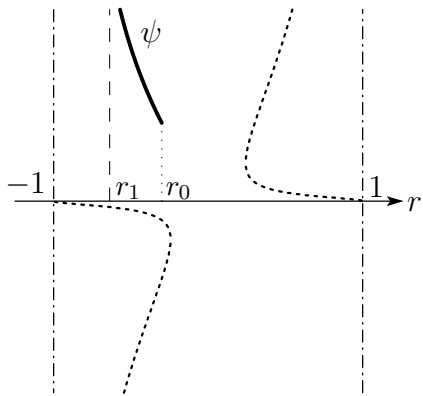

\centering
\scalebox{1.0}{\unitlength 0.1in%
%
}
\caption{The behavior of the graph of $\psi$ in Lemma \ref{thm:psi-shape3}}
\label{eta}
\end{figure}

\begin{lemma}\label{thm:psi-shape4}
If there exists $r_0\in\left(-1,a\right)$ with $\eta_1(r_0)<\psi(r_0)<\eta_2(r_0)$, then
\begin{equation*}
\lim_{r\downarrow -1}\psi(r)=-\infty.
\end{equation*}
\end{lemma}

\begin{figure}[H]
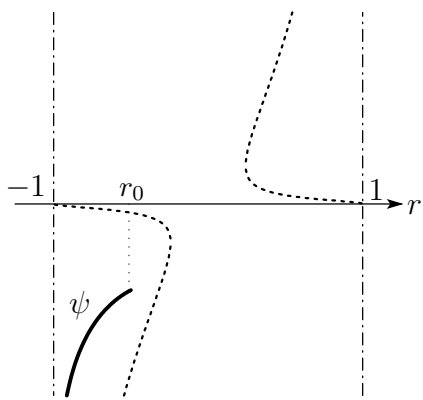

\centering
\scalebox{1.0}{\unitlength 0.1in%
%
}
\caption{The behavior of the graph of $\psi$ in Lemma \ref{thm:psi-shape4}}
\label{eta}
\end{figure}

By Lemmata \ref{thm:sign}-\ref{thm:psi-shape4}, we obtain the following proposition for the behavior of the graph of the solution $\psi$ of (\ref{eq:isopara-soliton-sphere3}).

\begin{proposition}\label{thm:graph-psi}
For the solution $\psi$ of the equation {\rm (\ref{eq:isopara-soliton-sphere3})}, the behavior of the graph of $\psi$ is like one of those defined by {\rm Figures \ref{psiex1}}-{\rm\ref{psiex5}}.
\end{proposition}

\begin{figure}[H]
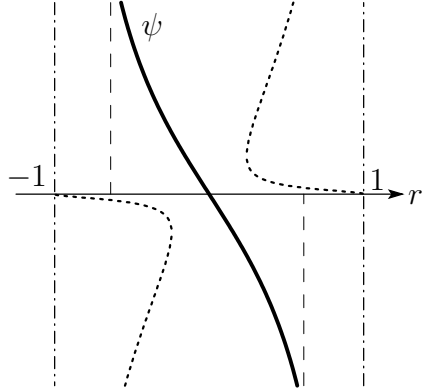

\centering
\scalebox{1.0}{\unitlength 0.1in%
}
\caption{The graph of $\psi$ (Type I)}
\label{psiex1}
\end{figure}

\begin{figure}[H]
\centering
\begin{minipage}{0.45\columnwidth}
\centering
\scalebox{1.0}{\unitlength 0.1in%
%
}
\caption{The graph of $\psi$ (Type II)}
\label{psiex2}
\end{minipage}
\hspace{0.5cm}
\begin{minipage}{0.45\columnwidth}
\centering
\scalebox{1.0}{\unitlength 0.1in%
%
}
\caption{The graph of $\psi$ (Type III)}
\label{psiex3}
\end{minipage}
\end{figure}

\begin{figure}[H]
\centering
\begin{minipage}{0.45\columnwidth}
\centering
\scalebox{1.0}{\unitlength 0.1in%
%
}
\caption{The graph of $\psi$ (Type IV)}
\label{psiex4}
\end{minipage}
\hspace{0.5cm}
\begin{minipage}{0.45\columnwidth}
\centering
\scalebox{1.0}{\unitlength 0.1in%
%
}
\caption{The graph of $\psi$ (Type V)}
\label{psiex5}
\end{minipage}
\end{figure}

For the graph of $\psi$ in Proposition \ref{thm:graph-psi}, we have not yet show whether $\psi$ in the case of {\rm Figures \ref{psiex4}} and {\rm \ref{psiex5}} exists or not.
From the following lemma, we obtain the existence.

\begin{lemma}
The solution $\psi$ of the equation {\rm (\ref{eq:isopara-soliton-sphere3})} in {\rm Figures \ref{psiex4}} and {\rm Figure \ref{psiex5}} exists.
\end{lemma}

\begin{proof}
For the set $S$ of all solutions of the equation {\rm (\ref{eq:isopara-soliton-sphere3})}, we define sets $S_1$, $S_2$, $S_3\subset S$  by
\begin{align*}
&S_1:=\{\psi\in S\vert\exists r_0\in(-1,1):\psi(r_0)=0\}\\
&S_2:=\{\psi\in S\vert\exists r_0\in(-1,1):\psi(r_0)=\eta_1(r_0)~or~\psi(r_0)=\eta_2(r_0)\}\\
&S_3:=\{\psi\in S\vert\psi(1)=0~or~\psi(-1)=0\}.
\end{align*}
Then, we have
\begin{equation*}
(-1,1)\times\mathbb{R}=\cup_{\psi\in S_1\cup S_2\cup S_3}{\rm Im}(\psi).
\end{equation*}
Since $\cup_{\psi\in S_1}{\rm Im}(\psi)$ and  $\cup_{\psi\in S_2}{\rm Im}(\psi)$ are open sets and $(-1,1)\times\mathbb{R}$ is connected, we find $S_3$ is not empty set and we obtain the statement of this lemma.
\end{proof}

Define $\zeta_1$ and $\zeta_2$ by $\zeta_i(r)=\frac{1}{k\sqrt{1-r^2}}\eta_i(r)$,~($i=1,2$).
By Proposition \ref{thm:graph-psi}, we obtain the following proposition.
\begin{proposition}\label{thm:graph-derivative}
For the solution $V$ of the equation {\rm (\ref{eq:isopara-soliton-sphere2})}, the behavior of the graph of $V'$ is like one of those defined by {\rm Figures \ref{V'ex1}}-{\rm \ref{V'ex5}}.
Here, the dotted curves in {\rm Figures \ref{V'ex1}}-{\rm \ref{V'ex5}} are the graphs of $\zeta_1$ and $\zeta_2$.

\begin{figure}[H]
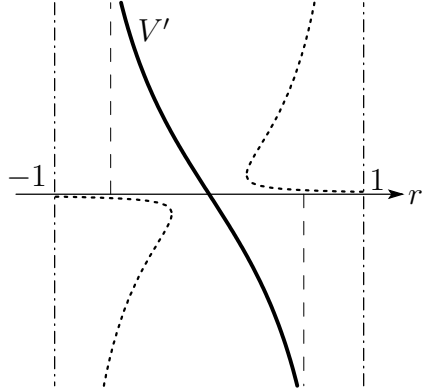

\centering
\scalebox{1.0}{\unitlength 0.1in%
}
\caption{The graph of $V'$ (Type I)}
\label{V'ex1}
\end{figure}

\begin{figure}[H]
\centering
\begin{minipage}{0.45\columnwidth}
\centering
\scalebox{1.0}{\unitlength 0.1in%
%
}
\caption{The graph of $V'$ (Type II)}
\label{V'ex2}
\end{minipage}
\hspace{0.5cm}
\begin{minipage}{0.45\columnwidth}
\centering
\scalebox{1.0}{\unitlength 0.1in%
%
}
\caption{The graph of $V'$ (Type II$'$)}
\label{V'ex2-2}
\end{minipage}
\end{figure}

\begin{figure}[H]
\centering
\begin{minipage}{0.45\columnwidth}
\centering
\scalebox{1.0}{\unitlength 0.1in%
%
}
\caption{The graph of $V'$ (Type II$''$)}
\label{V'ex2-3}
\end{minipage}
\hspace{0.5cm}
\begin{minipage}{0.45\columnwidth}
\centering
\scalebox{1.0}{\unitlength 0.1in%
%
}
\caption{The graph of $V'$ (Type III)}
\label{V'ex3}
\end{minipage}
\end{figure}

\begin{figure}[H]
\centering
\begin{minipage}{0.45\columnwidth}
\centering
\scalebox{1.0}{{\unitlength 0.1in%
%
}}
\caption{The graph of $V'$ (Type III$'$)}
\label{V'ex3-2}
\end{minipage}
\hspace{0.5cm}
\begin{minipage}{0.45\columnwidth}
\centering
\scalebox{1.0}{{\unitlength 0.1in%
%
}}
\caption{The graph of $V'$ (Type III$''$)}
\label{V'ex3-3}
\end{minipage}
\end{figure}

\begin{figure}[H]
\centering
\begin{minipage}{0.45\columnwidth}
\centering
\scalebox{1.0}{\unitlength 0.1in%
%
}
\caption{The graph of $V'$ (Type IV)}
\label{V'ex4}
\end{minipage}
\hspace{0.5cm}
\begin{minipage}{0.45\columnwidth}
\centering
\scalebox{1.0}{\unitlength 0.1in%
%
}
\caption{The graph of $V'$ (Type V)}
\label{V'ex5}
\end{minipage}
\end{figure}
\end{proposition}

\begin{proof}
For the solution $V$ of the equation {\rm (\ref{eq:isopara-soliton-sphere2})}, we have $V'(r)=\frac{1}{k\sqrt{1-r^2}}\psi(r)$ and $\psi$ is the solution of the equation {\rm (\ref{eq:isopara-soliton-sphere3})}.
Therefore, when the graph of $\psi$ is like one of those defined by Figure \ref{psiex1}, \ref{psiex4} and \ref{psiex5}, it is clear that the graph of $V'$ is like one defined by Figure \ref{V'ex1}, \ref{V'ex4} and \ref{V'ex5} respectively.
In the case where the graph of $\psi$ is like one of those defined by Figure \ref{psiex2} and \ref{psiex3}, there exists $r_0\in(-1,a]\cup[b,1)$ with $\psi(r_0)=\eta_1(r_0)$ or $\psi(r_0)=\eta_2(r_0)$ and we find $\psi'(r_0)=0$ by Lemma \ref{thm:sign}.
Then, we obtain $V''(r_0)=\frac{r_0}{k(1-r_0^2)^{\frac{3}{2}}}\psi(r_0)+\frac{1}{k\sqrt{1-r_0^2}}\psi'(r_0)=\frac{r_0}{k(1-r_0^2)^{\frac{3}{2}}}\psi(r_0)$.
Therefore, when the graph of $\psi$ is like one defined by Figure \ref{psiex2}, if $r_0>0$, then $V''(r_0)>0$ and the graph of $V'$ is like one defined by Figure \ref{V'ex2}, if $r_0=0$, then $V''(r_0)=0$ and the graph of $V'$ is like one defined by Figure \ref{V'ex2-2}, and if $r_0<0$, then $V''(r_0)<0$ and the graph of $V'$ is like one defined by Figure \ref{V'ex2-3}.
Also, when the graph of $\psi$ is like one defined by Figure \ref{psiex3}, if $r_0<0$, then $V''(r_0)>0$ and the graph of $V'$ is like one defined by Figure \ref{V'ex3}, if $r_0=0$, then $V''(r_0)=0$ and the graph of $V'$ is like one defined by Figure \ref{V'ex3-2}, and if $r_0>0$, then $V''(r_0)<0$ and the graph of $V'$ is like one defined by Figure \ref{V'ex3-3}.
In the case $k=1$, $3$, $6$, we find $a<R=0<b$.
Therefore, when the graph of $\psi$ is like one of those defined by Figure \ref{psiex2} and Figure \ref{psiex3}, the graph of $V'$ is like one defined by Figure \ref{V'ex2} and Figure \ref{V'ex3} respectively if $k=1$, $3$, $6$.
\end{proof}

By Proposition \ref{thm:graph-derivative}, we obtain Theorem \ref{thm:shape-graph}.
For the solution $V$ of the equation {\rm (\ref{eq:isopara-soliton-sphere2})}, when the graph of $V'$ is like one of those defined by Figure \ref{V'ex1}, \ref{V'ex4} and \ref{V'ex5}, it is clear that the graph of $V$ is like one defined by Figure \ref{Vex1}, \ref{Vex4} and \ref{Vex5} respectively.
When the graph of $V'$ is like one of those defined by Figure \ref{V'ex2}, \ref{V'ex2-2} and \ref{V'ex2-3}, the graph of $V$ is like one defined by Figure \ref{Vex2}.
Also, when the graph of $V'$ is like one of those defined by Figure \ref{V'ex3}, \ref{V'ex3-2} and \ref{V'ex3-3}, the graph of $V$ is like one defined by Figure \ref{Vex3}.

\section*{Acknowledgement}

\indent I would like to thank my supervisor Naoyuki Koike for helpful support and valuable comment.

\end{document}